\documentclass[a4paper,reqno]{amsart}
\usepackage{amsmath,amssymb,amsthm,graphicx,mathrsfs,url}
\usepackage[usenames,dvipsnames]{color}
\definecolor{darkred}{rgb}{0.4,0.1,0.1}
\definecolor{darkblue}{rgb}{0.1,0.1,0.4}

\usepackage{tikz}
\usetikzlibrary{datavisualization}
\usetikzlibrary{datavisualization.formats.functions}
\usetikzlibrary{hobby,backgrounds,patterns}
\usepackage[normalem]{ulem}

\numberwithin{equation}{section}
\theoremstyle{plain}% default

\newtheorem{theorem}{Theorem}[section]
\newtheorem{lemma}[theorem]{Lemma}

\newtheorem{proposition}[theorem]{Proposition}
\newtheorem{corollary}[theorem]{Corollary}

\theoremstyle{remark}
\newtheorem{remark}[theorem]{Remark}

\theoremstyle{definition}

\newtheorem{definition}[theorem]{Definition}

\newcommand\cH{\mathcal H}

\newcommand\eps{\varepsilon}

\DeclareMathOperator{\diver}{div}

\definecolor{darkgreen}{rgb}{0.1,0.45,0.1}
\definecolor{darkblue}{rgb}{0.1,0.1,0.4}
\definecolor{darkgrey}{rgb}{0.5,0.5,0.5}

\definecolor{darkred}{rgb}{0.6,0.0,0.0}

\newcommand\void[1]{}

\def\eps{\varepsilon}

\renewcommand{\phi}{\varphi}

\def\sa{\mathfrak a}

   \def\cH{{\mathcal H}}   
      
      \def\cO{{\mathcal O}}

\def\R{\mathbb{R}}
\def\C{\mathbb{C}}

\newcommand{\dom}{\mathrm{dom}\,}

\newcommand{\e}{\textup{e}}

\allowdisplaybreaks

\def\dd{{\,\mathrm d}}

\newcounter{counter_a}

\title[A new approach to the hot spots conjecture]{A new approach to the hot spots conjecture}

\author[J.~Rohleder]{Jonathan Rohleder}
\address{Matematiska institutionen \\ Stockholms universitet \\
106 91 Stockholm \\
Sweden}
\email{jonathan.rohleder@math.su.se}

\begin{document}

\begin{abstract}
We introduce a new variational principle for the study of eigenvalues and eigenfunctions of the Laplacians with Neumann and Dirichlet boundary conditions on planar domains. In contrast to the classical variational principles, its minimizers are gradients of eigenfunctions instead of the eigenfunctions themselves. This variational principle enables us to give an elementary analytic proof of the famous hot spots conjecture for the class of so-called lip domains. More specifically, we show that each eigenfunction corresponding to the lowest positive eigenvalue of the Neumann Laplacian on such a domain is strictly monotonous along two mutually orthogonal directions. In particular, its maximum and minimum  may only be located on the boundary.
\end{abstract}

\maketitle

\section{Introduction}

The aim of this paper is to establish a new approach to the famous hot spots conjecture and to provide an elementary analytic proof of this conjecture for a  class of planar domains. The crucial novelty is a non-standard variational principle for the non-zero eigenvalues of the Laplacians with Neumann and Dirichlet boundary conditions on simply connected planar domains. 

Variational principles are a classical tool in the analysis of eigenvalues and corresponding eigenfunctions of partial differential operators. Consider, for instance, the Neumann Laplacian $- \Delta_{\rm N}$ on a bounded, connected, sufficiently regular open set $\Omega \subset \R^2$, i.e.\ the Laplacian defined on functions on $\Omega$ whose derivative in the direction of the outer unit normal vector $\nu (x)$ vanishes for allmost all points $x$ on the boundary $\partial \Omega$. The operator $- \Delta_{\rm N}$ is self-adjoint in $L^2 (\Omega)$ with a purely discrete spectrum and its lowest eigenvalue is zero with the constant functions as its corresponding eigenfunctions. Its second (and lowest non-trivial) eigenvalue can be expressed as
\begin{align}\label{eq:minMaxclassic}
 \mu_2 = \min_{\substack{\psi \in H^1 (\Omega) \setminus \{0\} \\ \int_\Omega \psi = 0}} \frac{\int_\Omega |\nabla \psi|^2}{\int_\Omega |\psi|^2},
\end{align}
and a function $\psi$ in the Sobolev space $H^1 (\Omega)$ with vanishing integral is a minimizer if and only if $\psi$ is an eigenfunction of $- \Delta_{\rm N}$ corresponding to $\mu_2$. 

The description \eqref{eq:minMaxclassic}, as well as its counterparts for higher eigenvalues, other boundary conditions or more general elliptic differential operators, have proven to be indispensable tools in spectral theory. However, they are less suitable for the study of, e.g., critical points or further properties of an eigenfunction that are naturally closely related to its partial derivatives. This inspired us to establish an alternative variational description of the eigenvalues of the Laplacians with Neumann or Dirichlet boundary conditions on simply connected planar domains. Its most striking special case is the following expression for the first non-trivial eigenvalue $\mu_2$ of the Neumann Laplacian. The result is discussed in full detail in Section~\ref{sec:var}, see in particular Theorem \ref{thm:minMaxDN}, and we suspect that it may be useful also in connection with other questions in spectral geometry.

\begin{theorem}\label{thm:minMaxIntro}
Assume that $\Omega \subset \R^2$ is a bounded, simply connected Lipschitz domain with piecewise $C^\infty$-smooth boundary and that all its corners, if any, are convex. Then
\begin{align}\label{eq:minPrincipleNeumannIntro}
 \mu_2 = \min_{u = \binom{u_1}{u_2} \in \cH_{\rm N} \setminus \{0\}} \frac{\int_\Omega \left( |\nabla u_1|^2 + |\nabla u_2|^2 \right) - \int_{\partial \Omega} \kappa  \big( |u_1|^2 + |u_2|^2 \big)}{\int_\Omega \left( |u_1|^2 + |u_2|^2 \right)},
\end{align}
where $\cH_{\rm N}$ consists of all vector fields with components in the Sobolev space $H^1 (\Omega)$ such that their traces satisfy $\langle u, \nu \rangle = 0$ a.e.\ on $\partial \Omega$, and $\kappa$ is the signed curvature on $\partial \Omega$ w.r.t.\ the outer unit normal, defined on all boundary points except corners. The minimizers of \eqref{eq:minPrincipleNeumannIntro} are precisely the gradients of eigenfunctions $\psi$ of the Neumann Laplacian corresponding to $\mu_2$.
\end{theorem}

We point out that the curvature $\kappa (x)$ is well-defined for almost all $x \in \partial \Omega$, where only possible corners are excluded. Then $\kappa$ is a bounded function on $\partial \Omega$ with jumps at the corners. Due to a capacity argument, the behavior at the (at most finitely many) corners is not of any significance for our purposes; cf.\ the proof of Lemma \ref{lem}. Note that the sign of $\kappa$ is chosen such that $\kappa$ is negative locally where $\Omega$ is convex. 

Theorem \ref{thm:minMaxIntro} is actually a special case of Theorem \ref{thm:minMaxDN} below, where variational expressions for all non-trivial eigenvalues of the Neumann and Dirichlet Laplacians are derived. The crucial idea in the proof is to establish a self-adjoint operator that acts on vector fields in $L^2 (\Omega; \C^2)$ componentwise as the Laplacian and whose eigenvalues are exactly those of the Neumann and Dirichlet Laplacians except zero; its corresponding eigenfunctions can be obtained from the Neumann and Dirichlet Laplacian eigenfunctions by taking gradients. The assumption that $\Omega$ is simply connected comes in naturally through the use of the so-called Helmholtz-decomposition of the space $L^2 (\Omega; \C^2)$.

Our second aim is to apply the new variational principle to the famous hot spots conjecture. This conjecture suggests that a (any) eigenfunction of the Neumann Laplacian on a bounded domain corresponding to its first positive eigenvalue takes its maximum and minimum (only) on the boundary. This would imply that the hottest and coldest spots in an insulated medium with a ``generic'' initial heat distribution converge to the boundary for large time; see, e.g., the discussion in \cite{BB99}. 

The mathematical formulation of the hot spots conjecture seems to have been mentioned first by Rauch at a conference in 1974 and has in written form appeared first in Kawohl's book \cite{K85}. Most results known so far deal with the planar case $d = 2$. Ba\~nuelos and Burdzy \cite{BB99} proved that the conjecture is true for obtuse triangles and sufficiently long convex domains with symmetries, and Atar and Burdzy~\cite{AB04} showed it for so-called lip-domains, i.e.\ domains enclosed by the graphs of two Lipschitz continuous functions with Lipschitz constants at most one. Jerison and Nadirashvili proved it for a class of domains symmetric w.r.t.\ both coordinate axes \cite{JN00}. On the other hand Burdzy and Werner \cite{BW99} and Burdzy \cite{B05} constructed counterexamples given by certain multiply-connected domains, see also the numerical study \cite{K21}. It remains open whether the conjecture is true for all convex or even all simply connected domains. The most recent advances include a prove for general triangles by Judge and Mondal \cite{JM20,JM20err}, which was preceeded by Siudeja's partial result \cite{S15}. Krej\v{c}i\v{r}\'ik and Tu\v{s}ek \cite{KT19} studied the conjecture for thin curved strips and Steinerberger \cite{S20} showed that the maxima and minima of second eigenfunctions of the Neumann Laplacian on a convex domain must at least be close, in a specified sense, to the boundary.

It seems that most proofs of previous results on the hot spots conjecture are either of probabilistic nature, dealing with reflected Brownian motion \cite{AB04,BB99,BB00,B05,BW99,P02,S20} or are specific to particular shapes or symmetries \cite{JM20,JN00,KT19,S15}. Based on our Theorem~\ref{thm:minMaxIntro} it is now possible to give a purely analytic and elementary alternative proof of the hot spots conjecture for lip domains. Given an eigenfunction $\psi$ corresponding to $\mu_2$, by following the classical Courant argument of taking (component-wise) absolut values of minimizers and applying the minimum principle for superharmonic functions we deduce with the help of Theorem \ref{thm:minMaxIntro} that the components of $\nabla \psi$ are both positive inside $\Omega$, after a suitable rotation of the domain. This requires, however, that $(|\partial_1 \psi|, |\partial_2 \psi|)^\top$ belongs to the function class $\cH_{\rm N}$ of Theorem \ref{thm:minMaxIntro}, which is the case if and only if $\Omega$ is a lip domain, i.e.\ a Lipschitz domain bounded by the graphs of two Lipschitz continuous functions with Lipschitz constant less or equal one; see Definition \ref{def:lip} below.  This leads to a simple proof of the following result from \cite{AB04}. 

\begin{theorem}\label{thm:main}
Assume that $\Omega$ is a lip domain with piecewise $C^\infty$-smooth boundary whose corners are convex. Then the following assertions hold.
\begin{enumerate}
 \item If $\Omega$ is not a square then the first positive eigenvalue $\mu_2$ of the Neumann Laplacian on $\Omega$ is simple, i.e.\ the corresponding eigenfunction $\psi$ is unique up to scalar multiples.
 \item If $\Omega$ is not a rectangle then $\psi$ may be chosen such that its directional derivatives in both directions $\mathbf{e}_1 + \mathbf{e}_2$ and $\mathbf{e}_1 - \mathbf{e}_2$ are positive inside $\Omega$, where $\mathbf{e}_1$ and $\mathbf{e}_2$ are the standard basis vectors in $\R^2$. In particular, $\psi$ does not have any critical point inside $\Omega$ and, hence, takes its maximum and minimum on $\partial \Omega$ only.
\end{enumerate}
\end{theorem}

We point out that a basis of the eigenspace consisting only of real-valued functions may be chosen, and we assume implicitly that $\psi$ is real. It should be mentioned that the main result of \cite{AB04} holds for all lip domains without additional regularity assumptions; in our approach this may be resolved through approximation by more regular lip domains, but this is not the main focus of the present article. For Theorem \ref{thm:minMaxIntro} to hold in this form it is necessary that the domain of the Laplacian with Neumann boundary conditions possesses $H^2$ Sobolev regularity, see Remark \ref{rem}; this fails in the presence of non-convex corners.

The class of lip domains may be considered optimal for Theorem \ref{thm:main} (ii) to hold. In fact, the theorem implies that for these domains none of the partial derivatives of $\psi$ has any proper sign change on the closure of $\Omega$. On the other hand, if $\Omega$ is not a lip domain but sufficiently regular, then it may be rotated such that there exist two subsets of $\partial \Omega$ of positive length such that on one of them all normal vectors belong to the first and on the other to the second open quadrant, say. In this case, the Neumann boundary condition $\langle\nabla \psi, \nu\rangle = 0$ constantly on $\partial \Omega$ would imply that one of the components of $\nabla \psi$ must change sign in between---alternatively, $\psi$ is constant on a nontrivial part of the boundary, which should be possible only in rare cases such as on rectangles. We point out, however, that the assertion (ii) of Theorem \ref{thm:main} is stronger than the actual hot spots conjecture, for the given class of domains. Therefore it is still likely that the latter may hold true for any convex or maybe even simply connected domain.

To give a quick overview on the contents of the article, Section \ref{sec:prel} contains preliminary material on certain spaces of vector fields and reviews some abstract operator theory needed. The aim of Section \ref{sec:var} is to prove the variational principle of which Theorem \ref{thm:minMaxIntro} is a special case. Finally, in Section \ref{sec:hotSpots} this principle is applied to the hot spots conjecture, resulting in an elementary proof of Theorem \ref{thm:main}.

\section{Preliminaries}\label{sec:prel}

In this section we fix some notation and provide some preliminary material.

\subsection{Function spaces and Laplacians}

Let for now $\Omega \subset \R^2$ be any bounded, connected Lipschitz domain, see, e.g., \cite[Chapter 3]{McL}. For each integer $k \geq 1$ we denote by $H^k (\Omega)$ the $L^2$-based Sobolev space of order $k$ on $\Omega$. On the boundary $\partial \Omega$ we will make use of the Sobolev space $H^{1/2} (\partial \Omega)$ of order $1/2$ and its dual space $H^{- 1/2} (\partial \Omega)$. We denote the sesquilinear duality between $H^{- 1/2} (\partial \Omega)$ and $H^{1/2} (\partial \Omega)$ by $(\cdot, \cdot)_{\partial \Omega}$; in particular, if $\psi \in L^2 (\partial \Omega) \subset H^{- 1/2} (\partial \Omega)$ then
\begin{align*}
 (\psi, \xi)_{\partial \Omega} = \int_{\partial \Omega} \psi \overline{\xi}, \qquad \xi \in H^{1/2} (\partial \Omega).
\end{align*}
Recall that the trace map
\begin{align*}
 C^\infty (\overline{\Omega}) \ni u \mapsto u |_{\partial \Omega}
\end{align*}
extends uniquely to a bounded, everywhere defined, surjective operator $H^1 (\Omega) \to H^{1/2} (\partial \Omega)$; for each $u \in H^1 (\Omega)$ we write $u |_{\partial \Omega}$ for its trace. If $u$ belongs to $H^1 (\Omega)$ and $\Delta u$, taken in the distributional sense, belongs to $L^2 (\Omega)$, then we define its normal derivative $\partial_\nu u |_{\partial \Omega}$ to be the unique element in $H^{- 1/2} (\partial \Omega)$ such that 
\begin{align}\label{eq:Green}
 \int_\Omega \Delta u \overline{v} + \int_\Omega \langle \nabla u, \nabla v \rangle = \left( \partial_\nu u |_{\partial \Omega}, v |_{\partial \Omega} \right)_{\partial \Omega} \quad \text{for all}~v \in H^1 (\Omega),
\end{align}
see, e.g., \cite[Lemma 4.3]{McL}; here $\langle \cdot, \cdot \rangle$ denotes the Euclidean inner product in $\C^2$. We will refer to \eqref{eq:Green} as Green's first identity. If $u$ is slightly more regular, say $u \in H^2 (\Omega)$, then $\partial_\nu u |_{\partial \Omega} = \langle \nabla \psi |_{\partial \Omega}, \nu \rangle$, where $\nu$ denotes the outer unit normal vector field on $\partial \Omega$; the latter exists at almost every point of $\partial \Omega$ due to Rademacher's theorem.

The Neumann Laplacian $- \Delta_{\rm N}$ and the Dirichlet Laplacian $- \Delta_{\rm D}$ on $\Omega$ are defined as
\begin{align*}
 - \Delta_{\rm N} u & = - \Delta u, \quad \dom (- \Delta_{\rm N}) = \left\{u \in H^1 (\Omega) : \Delta u \in L^2 (\Omega), \partial_\nu u |_{\partial \Omega} = 0 \right\},
\end{align*}
and
\begin{align*}
 - \Delta_{\rm D} u & = - \Delta u, \quad \dom (- \Delta_{\rm D}) = \left\{u \in H^1 (\Omega) : \Delta u \in L^2 (\Omega), u |_{\partial \Omega} = 0 \right\}.
\end{align*}
Both are unbounded, self-adjoint operators in $L^2 (\Omega)$, and their spectra consist of isolated eigenvalues of finite multiplicities. Let
\begin{align}\label{eq:EVN}
 0 = \mu_1 < \mu_2 \leq \mu_3 \leq \dots 
\end{align}
be an enumeration of the eigenvalues of $- \Delta_{\rm N}$, counted with multiplicities, and let
\begin{align}\label{eq:EVD}
 \lambda_1 < \lambda_2 \leq \lambda_3 \leq \dots 
\end{align}
be the eigenvalues of $- \Delta_{\rm D}$, also these counted according to their multiplicities. As $\Omega$ is connected, the lowest eigenvalue $\mu_1$ of $- \Delta_{\rm N}$ is zero, with multiplicity one and corresponding eigenspace given by the constant functions. The lowest eigenvalue of $- \Delta_{\rm D}$ is positive and has multiplicity one as well.

Let us turn to spaces of vector fields, which will play a major role later on. Denote by $L^2 (\Omega; \C^2)$ the Hilbert space of vector fields on $\Omega$ whose both components belong to $L^2 (\Omega)$. We consider it together with the norm 
\begin{align*}
 \| u \|_{L^2 (\Omega; \C^2)} = \left( \int_\Omega \big(|u_1|^2 + |u_2|^2 \big) \right)^{1/2}, \quad u = \binom{u_1}{u_2} \in L^2 (\Omega; \C^2).
\end{align*}
Analogously, on $H^1 (\Omega; \C^2)$, the space of vector fields with components in $H^1 (\Omega)$, we consider the norm 
\begin{align*}
 \| u \|_{H^1 (\Omega; \C^2)} = \left( \|u\|_{L^2 (\Omega; \C^2)}^2 + \int_\Omega \big(|\nabla u_1|^2 + |\nabla u_2|^2 \big) \right)^{1/2}, \quad u = \binom{u_1}{u_2} \in H^1 (\Omega; \C^2).
\end{align*}
The space
\begin{align*}
 E (\Omega) = \left\{ u \in L^2 (\Omega; \C^2) : \diver u \in L^2 (\Omega) \right\},
\end{align*}
equipped with norm $(\|u\|_{L^2 (\Omega; \C^2)}^2 + \|\diver u\|_{L^2 (\Omega)}^2)^{1/2}$, is a Hilbert space and the trace mapping
\begin{align*}
 C^\infty (\overline{\Omega}; \C^2) \ni u \mapsto \left\langle u |_{\partial \Omega}, \nu \right\rangle 
\end{align*}
extends continuously to a bounded linear operator from $E (\Omega)$ onto $H^{-1/2} (\partial \Omega)$; see \cite[Chapter XIX, \S 1, Theorem 2]{DL}. In particular, the space
\begin{align*}
 H := \left\{ u \in L^2 (\Omega; \C^2) : \diver u = 0~\text{in}~\Omega, \left\langle u |_{\partial \Omega}, \nu \right\rangle = 0~\text{on}~\partial \Omega \right\}
\end{align*}
is well-defined. Moreover, $H$ is orthogonal in $L^2 (\Omega; \C^2)$ to the closed subspace $\nabla H^1 (\Omega)$, and the Helmholtz decomposition
\begin{align}\label{eq:Helmholtz}
 L^2 (\Omega; \C^2) = \nabla H^1 (\Omega) \oplus H
\end{align}
holds, see, e.g., \cite[Chapter XIX, \S 1, Theorem 4]{DL}. The space $H$ may be decomposed
\begin{align*}
 H = \nabla^\perp H_0^1 (\Omega) \oplus \left\{u \in H : \partial_2 u_1 = \partial_1 u_2~\text{in}~\Omega \right\},
\end{align*}
where $\nabla^\perp \phi = \binom{- \partial_2 \phi}{\partial_1 \phi}$. In particular, if $\Omega$ is simply connected then the Helmholtz decomposition \eqref{eq:Helmholtz} takes the form
\begin{align}\label{eq:goodDecomposition}
 L^2 (\Omega; \C^2) = \nabla H^1 (\Omega) \oplus \nabla^\perp H_0^1 (\Omega),
\end{align}
see, e.g., \cite[Lemma 2.10]{K10}.

\subsection{Sesquilinear forms and self-adjoint operators}\label{sec:Kato}

We review very briefly the relation between self-adjoint operators semi-bounded below and closed semi-bounded sesquilinear forms; for more details we refer the reader to \cite[Chapter VI]{Kato} or \cite[Chapters 10 and 12]{S12}.

Let $\cH$ be a Hilbert space with inner product $(\cdot, \cdot)$ and corresponding norm $\|\cdot\|$ and let $\sa$ be a sesquilinear form in $\cH$ with domain $\sa$; that is, $\dom \sa$ is a linear subspace of $\cH$ and $\sa : \dom \sa \times \dom \sa \to \C$ is linear in the first and anti-linear in the second entry. We say that $\sa$ is {\em symmetric} if
\begin{align*}
 \sa [u] := \sa [u, u] \in \R, \quad u \in \dom \sa,
\end{align*}
and {\em semi-bounded below} if, for some $\mu \in \R$,
\begin{align*}
 \sa [u] \geq \mu \|u\|^2, \quad u \in \dom \sa.
\end{align*}
If $\sa$ is a semi-bounded sesquilinear form then clearly
\begin{align*}
 \| u \|_\sa^2 := \sa [u] + (1 - \mu) \|u\|^2, \quad u \in \dom \sa,
\end{align*}
defines a norm on $\cH_\sa := \dom \sa$, and $\cH_\sa$ is a pre-Hilbert space, and $\sa$ is called {\em closed} if $\cH_\sa$ is a Hilbert space. The correspondence between closed semi-bounded sesquilinear forms and self-adjoint operators bounded below is summarized in the following proposition.

\begin{proposition}\label{prop:Kato}
Assume that $\sa$ is a symmetric, semi-bounded, closed sesquilinear form whose domain $\dom \sa$ is dense in $\cH$. Then the following assertions hold.
\begin{enumerate}
 \item There exists a unique self-adjoint, non-negative operator $A$ in $\cH$ with $\dom A \subset \dom \sa$ such that
\begin{align}\label{eq:repOp}
 (A u, v) = \sa [u, v], \quad u \in \dom A, v \in \dom \sa.
\end{align}
Moreover, $u \in \dom \sa$ belongs to $\dom A$ if and only if there exists $w \in \cH$ such that $\sa [u, v] = (w, v)$ holds for all $v \in \dom \sa$; in this case, $A u = w$.
 \item If, in addition, $\dom \sa$ is compactly embedded into $\cH$ then the spectrum of $A$ is purely discrete, i.e.\ it consists of isolated eigenvalues with finite multiplicities. Upon enumerating these eigenvalues non-decreasingly according to their multiplicities
\begin{align*}
 \eta_1 \leq \eta_2 \leq \dots,
\end{align*}
the min-max principle
\begin{align*}
 \eta_j = \min_{\substack{F \subset \cH_\sa~\text{subspace} \\ \dim F = j}} \;\; \max_{u \in F \setminus \{0\}} \frac{\sa [u]}{\|u\|^2}
\end{align*}
holds. In particular, the lowest eigenvalue $\eta_1$ of $A$ is given by
\begin{align}\label{eq:minPrinciple}
 \eta_1 = \min_{u \in \cH_\sa \setminus \{0\}} \frac{\sa [u]}{\|u\|^2}.
\end{align}
Moreover, $u$ is an eigenfunction of $A$ corresponding to $\eta_1$ if and only if $u$ minimizes \eqref{eq:minPrinciple}.
\end{enumerate}
\end{proposition}

\section{A non-standard variational principle for Neumann and Dirichlet Laplacian eigenvalues}\label{sec:var}

In this section we introduce a self-adjoint Laplacian acting on vector fields, whose eigenvalues coincide with the union of the positive eigenvalues of the Neumann and Dirichlet Laplacians, including multiplicities. This leads to a new variational characterization of these eigenvalues.

First we construct the desired Laplacian using a corresponding sesquilinear form; this construction works for any domain with piecewise smooth boundary, not necessarily simply connected. For the notions of semi-bounded, closed sesquilinear forms and their relation to semi-bounded self-adjoint operators see the short survey in Section \ref{sec:Kato} above. In the following we use the abbreviation
\begin{align}\label{eq:HN}
 \cH_{\rm N} := \left\{ u \in H^1 (\Omega; \C^2) : \left\langle u |_{\partial \Omega}, \nu \right\rangle = 0~\text{on}~\partial \Omega \right\}.
\end{align}
Moreover, we denote by $\kappa$ the signed curvature function along the piecewise smooth curve $\partial \Omega$ w.r.t.\ the outer unit normal $\nu$, defined on all points of $\partial \Omega$ except possible corners; in particular, $\kappa$ is bounded and piecewise smooth, with possible jumps at the corners. If $\Omega$ is convex then $\kappa (x) \leq 0$ holds for almost all $x \in \partial \Omega$.

\begin{proposition}\label{prop:form}
Assume that $\Omega \subset \R^2$ is a bounded domain with piecewise $C^\infty$-smooth boundary. Then the sesquilinear form $\sa$ in $L^2 (\Omega; \C^2)$ defined by
\begin{align*}
 \sa \left[ u, v \right] & = \int_\Omega \big( \left\langle\nabla u_1, \nabla v_1\right\rangle + \left\langle\nabla u_2, \nabla v_2\right\rangle \big) - \int_{\partial \Omega} \kappa \langle u, v \rangle, \quad u = \binom{u_1}{u_2}, v =  \binom{v_1}{v_2},
\end{align*}
with domain $\dom \sa = \cH_{\rm N}$ is symmetric, semi-bounded and closed, and $\dom \sa$ is dense in $L^2 (\Omega; \C^2)$. The self-adjoint operator in $L^2 (\Omega; \C^2)$ associated with $\sa$ as in Proposition \ref{prop:Kato} has a purely discrete spectrum bounded below.
\end{proposition}

\begin{proof}
First of all, 
\begin{align*}
 \sa \left[u \right] = \int_\Omega \left( |\nabla u_1|^2 + |\nabla u_2|^2 \right) - \int_{\partial \Omega} \kappa \big( |u_1|^2 + |u_2|^2 \big), \quad u = \binom{u_1}{u_2} \in \dom \sa,
\end{align*}
is real, i.e.\ $\sa$ is symmetric. Moreover, $\dom \sa$ is dense in $L^2 (\Omega; \C^2)$ as $C_0^\infty (\Omega; \C^2) \subset \dom \sa$. To see that $\sa$ is semi-bounded below, let us assume that 
\begin{align*}
 \kappa_0 := - \max_{\partial \Omega} \kappa < 0;
\end{align*}
otherwise one sees immediately that $\sa$ is non-negative. Recall that the boundedness of the trace map from $H^1 (\Omega)$ into $L^2 (\partial \Omega)$ may be sharpened in the following way, see, e.g., \cite[Lemma 4.2]{GM09}: for each $\eps > 0$ there exists some $C_\eps > 0$ such that
\begin{align}\label{eq:Ehrling}
 \int_{\partial \Omega} \big|\phi |_{\partial \Omega} \big|^2 \leq \eps \int_\Omega |\nabla \phi|^2 + C_\eps \int_\Omega |\phi|^2, \quad \phi \in H^1 (\Omega).
\end{align}
For $u \in \cH_{\rm N}$ we have
\begin{align*}
 \sa [u] & \geq \int_\Omega \left( |\nabla u_1|^2 + |\nabla u_2|^2 \right) + \kappa_0 \int_{\partial \Omega} \big( |u_1|^2 + |u_2|^2 \big),
\end{align*}
and if we choose $\eps > 0$ such that $1 + \eps \kappa_0 \geq 0$ it follows by applying \eqref{eq:Ehrling} to each component of $u$ that
\begin{align}\label{eq:first}
 \sa [u] & \geq (1 + \eps \kappa_0) \int_\Omega \left( |\nabla u_1|^2 + |\nabla u_2|^2 \right) + C_\eps \kappa_0 \int_{\Omega} \big( |u_1|^2 + |u_2|^2 \big).
\end{align}
Thus $\sa$ is semi-bounded below by some constant $\mu := C_\eps \kappa_0$. Moreover, it follows from \eqref{eq:first} and the boundedness of the trace operator that 
\begin{align*}
 \| u \|_\sa^2 = \sa [u] + (1 - \mu) \int_\Omega \big( |u_1|^2 + |u_2|^2 \big), \quad u \in \dom \sa,
\end{align*}
defines a norm which is equivalent to the norm of $H^1 (\Omega; \C^2)$ on $\dom \sa = \cH_{\rm N}$. Again by the continuity of the trace operator $H^1 (\Omega) \to H^{1/2} (\partial \Omega)$, $\cH_{\rm N}$ is a closed subspace of $H^1 (\Omega; \C^2)$ and, hence, $\sa$ is closed. Therefore there exists a unique self-adjoint operator $A$ in $L^2 (\Omega; \C^2)$ associated with $\sa$ as in \eqref{eq:repOp} with $\dom A \subset \dom \sa$. Finally, $\dom a \subset H^1 (\Omega; \C^2)$ is compactly embedded into $L^2 (\Omega; \C^2)$, which yields that $A$ has purely discrete spectrum bounded below by $\mu$. 
\end{proof}

We point out that the sesquilinear form $\sa$ in the previous proposition is obviously non-negative if $\Omega$ is convex (in which case $\kappa$ is non-positive). It is a consequence of Theorem \ref{thm:translateEV} below that this remains true if $\Omega$ is simply connected and sufficiently regular.

Our next aim is to establish a relation between the spectrum of the self-adjoint operator associated with $\sa$ and the spectra of the Dirichlet and Neumann Laplacians on $\Omega$. This will in particular lead to a variational principle for the union of their eigenvalues. If $\Omega$ is a Lipschitz domain with a piecewise $C^\infty$-smooth boundary and each corner is convex then
\begin{align}\label{eq:regularity}
 \dom (- \Delta_{\rm \bullet}) \subset H^2 (\Omega) \quad \text{for}~\bullet = \rm D, N,
\end{align}
holds for the domains of the Dirichlet and Neumann Laplacians, see, e.g., \cite[Section 8]{GM11}. For the following theorem we remind the reader that $\nabla^\perp = \binom{- \partial_2}{\partial_1}$.

\begin{theorem}\label{thm:translateEV}
Assume that $\Omega \subset \R^2$ is a bounded Lipschitz domain with piecewise $C^\infty$-smooth boundary whose corners are convex. Let $A$ be the self-adjoint operator in $L^2 (\Omega; \C^2)$ associated with the sesquilinear form $\sa$ in Proposition \ref{prop:form}. Moreover, let the eigenvalues of $- \Delta_{\rm N}$ be enumerated as in \eqref{eq:EVN} and let $\psi_1, \psi_2, \dots$ form an orthonormal basis of $L^2 (\Omega)$ such that $- \Delta_{\rm N} \psi_k = \mu_k \psi_k$ holds for $k = 1, 2, \dots$; analogously let the eigenvalues of $- \Delta_{\rm D}$ be enumerated as in \eqref{eq:EVD} and let $\phi_1, \phi_2, \dots$ be an orthonormal basis of $L^2 (\Omega)$ consisting of corresponding eigenfunctions, $- \Delta_{\rm D} \phi_k = \lambda_k \phi_k$ for all $k$. Then the following hold.
\begin{enumerate}
 \item For each $k \geq 2$, $\nabla \psi_k$ is nontrivial, belongs to $\dom A$, and satisfies $A \nabla \psi_k = \mu_k \nabla \psi_k$. Moreover, the functions $\frac{1}{\sqrt{\mu_k}} \nabla \psi_k$ form an orthonormal basis of $\nabla H^1 (\Omega)$.
 \item For each $k \geq 1$, $\nabla^\perp \phi_k$ is nontrivial, belongs to $\dom A$, and satisfies $A \nabla^\perp \phi_k = \lambda_k \nabla^\perp \phi_k$. Moreover, the functions $\frac{1}{\sqrt{\lambda_k}} \nabla^\perp \phi_k$ form an orthonormal basis of $\nabla^\perp H_0^1 (\Omega)$.
\end{enumerate}
In particular, if $\Omega$ is simply connected then the spectrum of $A$ coincides with the union of the positive eigenvalues of the Neumann and Dirichlet Laplacians, counted with multiplicities.
\end{theorem}

The proof of this theorem will be based on a computation summarized in the following lemma.

\begin{lemma}\label{lem}
Assume that $\Omega$ is a bounded Lipschitz domain with piecewise $C^\infty$-smooth boundary and that $u = (u_1, u_2)^\top \in \cH_{\rm N}$ has the following properties:
\begin{enumerate}
 \item[(a)] $u \in C^\infty (\Omega; \C^2)$ and for each open set $U \subset \Omega$ whose closure does not contain any corners of $\partial \Omega$, $u \in C^\infty (\overline{U}; \C^2)$;
 \item[(b)] $\partial_1 u_2 - \partial_2 u_1 = 0$ almost everywhere on $\partial \Omega$;
 \item[(c)] $\Delta u_j \in L^2 (\Omega)$, $j = 1, 2$.
\end{enumerate}
Then for each $v = (v_1, v_2)^\top \in \cH_{\rm N}$ the identity
\begin{align*}
 \sa [u, v] = - \int_\Omega \big( \Delta u_1 \overline{v_1} + \Delta u_2 \overline{v_2} \big)
\end{align*}
holds.
\end{lemma}

\begin{proof}
As $u_j \in H^1 (\Omega)$ with $\Delta u_j \in L^2 (\Omega)$, $j = 1, 2$, we may apply Green's identity \eqref{eq:Green} to see that for each $v \in \cH_{\rm N}$ we have 
\begin{align}\label{eq:partialInt}
\begin{split}
 \sa [u, v] + \int_\Omega \big( \Delta u_1 \overline{v_1} + \Delta u_2 \overline{v_2} \big) & = (\partial_\nu u_1 |_{\partial \Omega}, v_1 |_{\partial \Omega})_{\partial \Omega} + (\partial_\nu u_2  |_{\partial \Omega}, v_2  |_{\partial \Omega})_{\partial \Omega} \\
 & \quad - \int_{\partial \Omega} \kappa \langle u, v \rangle.
\end{split}
\end{align}
Our aim is therefore to show that the right-hand side of the latter equality vanishes. Let $\Gamma_l$ be a smooth arc in $\partial \Omega$ and let $v \in C^\infty (\overline \Omega; \C^2)$ be such that its trace $v |_{\partial \Omega}$ is supported compactly in the interior of $\Gamma_l$. (Note that for the moment we do not require $v \in \cH_{\rm N}$.) Due to the property (a) on $u$ we may then rewrite the dualities on the right-hand side of \eqref{eq:partialInt} as a boundary integral,
\begin{align*}
 (\partial_\nu u_1  |_{\partial \Omega}, v_1 |_{\partial \Omega})_{\partial \Omega} + (\partial_\nu u_2 |_{\partial \Omega}, v_2 |_{\partial \Omega})_{\partial \Omega} & = \int_{\Gamma_l} \left\langle \binom{\langle \nu, \nabla u_1 \rangle}{\langle \nu, \nabla u_2 \rangle}, v \right\rangle,
\end{align*}
and the latter integral may be expanded, using the unit tangential vector field $\tau$ on $\partial \Omega$,
\begin{align}\label{eq:abc}
\begin{split}
 \int_{\Gamma_l} \left\langle \binom{\langle \nu, \nabla u_1 \rangle}{\langle \nu, \nabla u_2 \rangle}, v \right\rangle & = \int_{\Gamma_l} \left\langle \binom{\langle \partial_1 u, v \rangle}{\langle \partial_2 u, v \rangle}, \nu \right\rangle \\
 & = \int_{\Gamma_l} \langle \tau, v \rangle \left\langle \binom{\langle \partial_1 u, \tau \rangle}{\langle \partial_2 u, \tau \rangle}, \nu \right\rangle + \int_{\Gamma_l} \langle \nu, v \rangle \left\langle \binom{\langle \partial_1 u, \nu \rangle}{\langle \partial_2 u, \nu \rangle}, \nu \right\rangle.
\end{split}
\end{align}
Let now $r : [0, L] \to \R^2$ be a $C^\infty$-smooth arc length parametrization of $\Gamma_l$, taken in the direction such that $\Omega$ is on the left. Then at each point $x = r (s)$ with $s \in (0, L)$,
\begin{align}\label{eq:curveProperties}
 \tau (x) = r' (s), \quad \nu (x) = - r' (s)^\perp, \quad \text{and} \quad \kappa (x) = \langle \nu (x), r'' (s) \rangle.
\end{align}
Then for $x = r (s)$, $s \in (0, L)$, we have 
\begin{align*}
 \frac{\dd}{\dd s} (u_j (r (s))) = \langle (\nabla u_j) (r (s)), r' (s)\rangle = \partial_1 u_j (x) \tau_1 (x) + \partial_2 u_j (x) \tau_2 (x),
\end{align*}
and therefore a straightforward computation gives
\begin{align*}
 \left\langle \binom{\langle \partial_1 u (x), \tau (x) \rangle}{\langle \partial_2 u (x), \tau (x) \rangle}, \nu (x) \right\rangle & = \left\langle \frac{\dd u (r (s))}{\dd s}, \nu (x) \right\rangle + \partial_1 u_2 (x) - \partial_2 u_1 (x).
\end{align*}
Proceeding with this and using the assumption (b) on $u$ as well as $\langle u, \nu \rangle = 0$ constantly on $\partial \Omega$, we get
\begin{align*}
 \left\langle \binom{\langle \partial_1 u (x), \tau (x) \rangle}{\langle \partial_2 u (x), \tau (x) \rangle}, \nu (x) \right\rangle & = \frac{\dd}{\dd s} \big\langle u (r (s)), \nu (r (s)) \big\rangle - \left\langle u (x), - r'' (s)^\perp \right\rangle \\
 & = \langle u (x), \tau (x) \rangle \left\langle \tau (x), r'' (s)^\perp \right\rangle = \kappa (x) \left\langle u (x), \tau (x) \right\rangle,
\end{align*}
where we have used \eqref{eq:curveProperties}. Therefore \eqref{eq:abc} yields
\begin{align}\label{eq:result}
 \int_{\Gamma_l} \left\langle \binom{\langle \nu, \nabla u_1 \rangle}{\langle \nu, \nabla u_2 \rangle}, v \right\rangle & = \int_{\Gamma_l} \kappa \langle \tau, v \rangle \langle u, \tau \rangle + \int_{\Gamma_l} \langle \nu, v \rangle \left\langle \binom{\langle \partial_1 u, \nu \rangle}{\langle \partial_2 u, \nu \rangle}, \nu \right\rangle.
\end{align}
This is true for each smooth arc $\Gamma_l$ in $\partial \Omega$ and each $v$ that is smooth up to the boundary and whose trace has compact support in the interior of $\Gamma_l$. 

Let now $v \in \cH_{\rm N}$ be arbitrary. By a capacity argument the functions in $C_0^\infty (\R^2)$ whose supports do not contain the corners of $\Omega$ are dense in $H^1 (\R^2)$, see \cite[Chapter VIII, Corollary 6.4]{EE87}. Hence it follows from \eqref{eq:result} by approximation and the continuity of the trace operator that
\begin{align*}
 (\partial_\nu u_1  |_{\partial \Omega}, v_1 |_{\partial \Omega})_{\partial \Omega} + (\partial_\nu u_2 |_{\partial \Omega}, v_2 |_{\partial \Omega})_{\partial \Omega} & = \int_{\partial \Omega} \kappa \langle \tau, v \rangle \langle u, \tau \rangle = \int_{\partial \Omega} \kappa \langle u, v \rangle.
\end{align*}
From this and \eqref{eq:partialInt} the assertion of the lemma follows.
\end{proof}

\begin{proof}[Proof of Theorem \ref{thm:translateEV}](i) Let first $\psi$ be an eigenfunction of $- \Delta_{\rm N}$ corresponding to a nonzero eigenvalue $\mu$ and let 
\begin{align*}
 u := \binom{u_1}{u_2} := \nabla \psi.
\end{align*}
Then $u$ satisfies all conditions of Lemma \ref{lem}. Indeed, \eqref{eq:regularity} gives $u \in H^1 (\Omega; \C^2)$, and the Neumann boundary condition imposed on $\psi$ reads $\langle u, \nu \rangle = 0$ almost everywhere on $\partial \Omega$, i.e.\ $u \in \cH_{\rm N}$. Moreover, applying the Laplacian componentwise in the distributional sense yields $\Delta u_j = \partial_j \Delta \psi = - \mu u_j \in L^2 (\Omega)$, $j = 1, 2$, so that condition (c) of the lemma is satisfied. By classical elliptic regularity theory, $\psi$ belongs to $C^\infty (\Omega)$ and is $C^\infty$-smooth up to the boundary, except possibly near the corners, see, e.g., \cite[Theorem 4.18 (ii)]{McL}. This yields condition (a) of the lemma. Moreover, $\partial_1 u_2 - \partial_2 u_1 = \partial_1 \partial_2 u - \partial_2 \partial_1 u = 0$ in $\Omega$, which implies that condition (b) holds. Hence
\begin{align*}
 \sa \left[ u, v \right] = - \int_\Omega \big( \Delta u_1 \overline{v_1} + \Delta u_2 \overline{v_2} \big) = \mu \int_\Omega \langle u, v \rangle
\end{align*}
holds for all $v \in \dom \sa$, which implies $u \in \dom A$ and $A u = \mu u$. Thus, as $\psi$ is not constant, $\mu$ is an eigenvalue of $A$ with corresponding eigenfunction $u$. 

Take the eigenfunctions $\psi_k, \psi_i$ of $- \Delta_{\rm N}$, where $k, i \neq 1$. Then $\mu_i \neq 0$ and
\begin{align*}
 \int_\Omega \langle \nabla \psi_k, \nabla \psi_i \rangle & = - \int_\Omega \psi_k \Delta \psi_i = \mu_i \int_\Omega \psi_k \psi_i,
\end{align*}
which equals $0$ if $k \neq i$ and $\mu_k$ if $k = i$. Thus the functions $\frac{1}{\sqrt{\mu_k}} \nabla \psi_k$ form an orthonormal system in $\nabla H^1 (\Omega)$. Let us assume there exists $u = \nabla \phi \in \nabla H^1 (\Omega)$ which is orthogonal to all vector fields $\nabla \psi_k$. Then
\begin{align*}
 0 = \int_\Omega \left\langle \nabla \phi, \nabla \psi_k \right\rangle & = \mu_k \int_\Omega \phi \psi_k,
\end{align*}
and, hence, $\phi$ is orthogonal to all non-constant eigenfunctions of $- \Delta_{\rm N}$. Consequently, $\phi$ is constant, i.e.\ $\nabla \phi = 0$. Therefore indeed the functions $\frac{1}{\sqrt{\mu_k}} \nabla \psi_k$ form an orthonormal basis of $\nabla H^1 (\Omega)$.

(ii) Let now $\phi \in \dom (- \Delta_{\rm D})$ be a nontrivial function such that $- \Delta_{\rm D} \phi = \lambda \phi$, and let
\begin{align*}
 u := \binom{u_1}{u_2} := \nabla^\perp \phi.
\end{align*}
Due to \eqref{eq:regularity}, $u \in H^1 (\Omega; \C^2)$. Furthermore, $u$ is constantly zero along $\partial \Omega$, which gives $\langle u, \nu \rangle = - \langle \nabla \phi, \tau \rangle = 0$ on $\partial \Omega$. Thus $u \in \cH_{\rm N}$. Moreover, $- \Delta u_j = \lambda u_j \in L^2 (\Omega)$ holds for $j = 1, 2$, and $\partial_2 u_1 - \partial_1 u_2 = - \Delta \phi = \lambda \phi$ vanishes on $\partial \Omega$. Since, again by elliptic regularity, $u$ is smooth up to the boundary, possibly except near corners, we may again apply Lemma \ref{lem} and obtain
\begin{align*}
 \sa \left[ u, v \right] = - \int_\Omega \big( \Delta u_1 \overline{v_1} + \Delta u_2 \overline{v_2} \big) = \lambda \int_\Omega \left\langle u, v \right\rangle
\end{align*}
for all $v \in \dom \sa$. Hence $u \in \dom A$ and $A u = \lambda u$. As $\phi$ is not constant, $u$ is nontrivial and thus is an eigenfunction of $A$. It is an argument entirely analogous to the one in (i) which proves that the functions $\frac{1}{\sqrt{\lambda_k}} \nabla^\perp \phi_k$ form an orthonormal basis in $\nabla^\perp H_0^1 (\Omega)$.

The remaining assertion follows from (i) and (ii) and the fact that the Helmholtz decomposition of $L^2 (\Omega; \C^2)$ takes the form \eqref{eq:goodDecomposition} if $\Omega$ is simply connected.
\end{proof}

\begin{remark}
In the context of the present problem we are mostly interested in the sesquilinear form $\sa$ rather than the associated operator $A$. However, the computations made in Lemma \ref{lem} indicate that $A$ acts compontentwise as the Laplacian and its domain consists of all $u \in \cH_{\rm N}$ such that $\Delta u_j$ belongs to $L^2 (\Omega)$, $j = 1, 2$, and $u$ is ``irrotational along the boundary'', i.e.\ $\partial_1 u_2 - \partial_2 u_1 = 0$ on $\partial \Omega$ in a weak sense. 
\end{remark}

\begin{remark}\label{rem}
In case $\Omega$ is simply connected the previous theorem implies that the operator $A$ admits a direct sum decomposition $A = A_1 \oplus A_2$, where $A_1$ and $A_2$ are self-adjoint operators in the mutually orthogonal spaces $\nabla H^1 (\Omega)$ and $\nabla^\perp H_0^1 (\Omega)$, respectively. Especially, in this case the spectrum of $A_1$ is given by the non-zero eigenvalues of $- \Delta_{\rm N}$ and the spectrum of $A_2$ consists of all eigenvalues of $- \Delta_{\rm D}$. From this one may deduce alternative variational principles for the eigenvalues of the two operators separately. We do not go further into these details here as this is not of any use for the primary goal of this article. However, we would like to point out that the direct sum decomposition of $A$ and its corresponding quadratic form $\sa$ relies crucially on the fact that eigenfunctions of $- \Delta_{\rm N}$ and $- \Delta_{\rm D}$ belong to the Sobolev space $H^2 (\Omega)$, which requires some regularity of $\partial \Omega$ and, in particular, excludes non-convex corners.
\end{remark}

As a direct consequence of the previous theorem we obtain the following variational principles for the Neumann and Dirichlet Laplacian eigenvalues. Recall that $\cH_{\rm N}$ is given in \eqref{eq:HN}.

\begin{theorem}\label{thm:minMaxDN}
Assume that $\Omega \subset \R^2$ is a bounded, simply connected domain with piecewise $C^\infty$-smooth boundary and that all corners are convex. Denote by
\begin{align*}
 \eta_1 \leq \eta_2 \leq \dots
\end{align*}
the union of the positive eigenvalues of $- \Delta_{\rm N}$ and $- \Delta_{\rm D}$, counted according to their multiplicities.\footnote{In other words, the $\eta_j$ are the positive eigenvalues of the direct sum $- \Delta_{\rm N} \oplus - \Delta_{\rm D}$ in $L^2 (\Omega) \oplus L^2 (\Omega)$.} Then
\begin{align}\label{eq:minMax}
 \eta_j = \min_{\substack{F \subset \cH_{\rm N}~\text{subspace} \\ \dim F = j}} \max_{u = \binom{u_1}{u_2} \in F \setminus \{0\}} \frac{\int_\Omega \left( |\nabla u_1|^2 + |\nabla u_2|^2 \right) - \int_{\partial \Omega} \kappa  \big( |u_1|^2 + |u_2|^2 \big)}{\int_\Omega \left( |u_1|^2 + |u_2|^2 \right)}.
\end{align}
Especially, the first positive eigenvalue $\mu_2$ of $- \Delta_{\rm N}$ is given by
\begin{align}\label{eq:minPrincipleNeumann}
 \mu_2 = \min_{u = \binom{u_1}{u_2} \in \cH_{\rm N} \setminus \{0\}} \frac{\int_\Omega \left( |\nabla u_1|^2 + |\nabla u_2|^2 \right) - \int_{\partial \Omega} \kappa  \big( |u_1|^2 + |u_2|^2 \big)}{\int_\Omega \left( |u_1|^2 + |u_2|^2 \right)}.
\end{align}
Moreover, if $\psi$ is an eigenfunction of $- \Delta_{\rm N}$ corresponding to $\mu_2$ then the minimum in \eqref{eq:minPrincipleNeumann} is attained at $u = \nabla \psi$, and, conversely, each minimizer $u$ of \eqref{eq:minPrincipleNeumann} satisfies $u = \nabla \psi$ for some $\psi \in \ker (- \Delta_{\rm N} - \mu_2)$.
\end{theorem}

\begin{proof}
By Theorem \ref{thm:translateEV} the $\eta_j$ are precisely the eigenvalues of the operator $A$, and its corresponding sesquilinear form is given in Proposition \ref{prop:form} and has $\cH_{\rm N}$ as its domain. From this the min-max principle \eqref{eq:minMax} follows immediately; cf.\ Proposition \ref{prop:Kato}. Moreover, \eqref{eq:minPrincipleNeumann} and the statement about the minimizers follow if we see that $\eta_1 = \mu_2$ and $\ker (A - \eta_1) = \nabla \ker (- \Delta_{\rm N} - \mu_2)$. But this follows directly from the well-known inequality $\mu_2 < \lambda_1$, see \cite{P52,F05}.
\end{proof}

\section{Application to the hot spots conjecture}\label{sec:hotSpots}

In this section we apply the variational principle obtained in Theorem \ref{thm:minMaxDN} to give an elementary and purely analytic proof of the hot spots conjecture for a class of planar domains. The following notion was introduced by Burdzy and Chen in \cite{BC98}.

\begin{definition}\label{def:lip}
A bounded Lipschitz domain $\Omega \subset \R^2$ is called {\em lip domain} if
\begin{align*}
 \Omega = \left\{ (x, y)^\top: f_1 (x) < y < f_2 (x), x \in (a, b) \right\},
\end{align*}
where $f_1, f_2 : [a, b] \to \R$ are Lipschitz continuous functions with Lipschitz constant at most one such that $f_1 (x) < f_2 (x)$ for all $x \in (a, b)$, $f_1 (a) = f_2 (a)$ and $f_1 (b) = f_2 (b)$.
% We say that {\em $\Omega$ faces two opposite quadrants} if, possibly after a rotation, $\nu (x)$ belong to $Q_2 \cup Q_4$,
% %= \{(x, y)^\top \in \R^2 : x y \leq 0 \}$
% the union of the second and fourth closed quadrants in $\R^2$, for almost all $x \in \partial \Omega$.
\end{definition}

\begin{figure}[h]
\begin{center}
\begin{tikzpicture}[rotate=-45,transform shape]
\coordinate (a) at (6,-1);
\coordinate (b) at (10, -0.5);
\coordinate (c) at (11,-0.1);
\coordinate (d) at (11,3.24);
\coordinate (e) at (11,1);
\draw[thick] (a) edge (b);
\draw[thick] (e) edge (d);
\node[rotate=45] at (8.5,0.6) {$\Omega$};
\draw[scale=1,domain=6:11,smooth,variable=\x,thick] plot ({\x},{sqrt(\x - 6) + 1});
\draw[scale=1,domain=10:11,smooth,variable=\x,thick] plot ({\x},{1.5*(sqrt(\x - 10) - 0.5/1.5)});
\draw[thick] (a) edge (6,1);
\end{tikzpicture}
\end{center}
\caption{A lip domain.}
\label{fig:domains}
\end{figure}
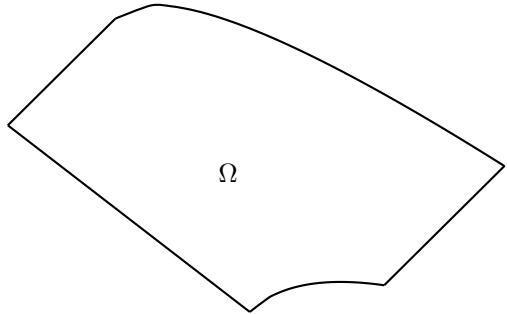

Figure~\ref{fig:domains} depicts an example of a lip domain, see also Figure~\ref{fig:cor} below. To name a few simple cases, right and obtuse triangles belong to the class of lip domains, while acute triangles don't; right and obtuse trapezoids (in particular all parallelograms) belong to it, while acute trapezoids don't. It is an immediate consequence of the definition that each lip domain is simply connected.

Before proceeding to the proof of Theorem \ref{thm:main} we point out that the requirement on the functions $f_1, f_2$ that constitute the boundary of a lip domain to have Lip\-schitz constant at most one is equivalent to the outer normal vector $\nu (x)$ having its angle with the positive real axis in $[\frac{\pi}{4}, \frac{3 \pi}{4}] \cup [\frac{5 \pi}{4}, \frac{7 \pi}{4}]$ for almost all $x \in \partial \Omega$. Hence, rotating any lip domain by $\pi/4$ in positive direction leads to a domain for which $\nu (x)$ belongs to $Q_2 \cup Q_4$, the union of the second and fourth closed quadrant in the plane, for almost all $x \in \partial \Omega$. Conversely, any simply connected Lipschitz domain with the latter property can be rotated into a lip domain; cf.\ also the discussion in \cite[Proposition 7.4]{JM22}. Now we prove Theorem \ref{thm:main}.

\begin{proof}[Proof of Theorem \ref{thm:main}]
To prove Theorem \ref{thm:main}, let us assume that $\Omega$ is rotated such that the outer unit normal $\nu (x) = (\nu_1 (x), \nu_2 (x))^\top$ belongs to $Q_2 \cup Q_4$ for almost all $x \in \partial \Omega$, see the discussion above this proof. Let $\psi$ be an arbitrary eigenfunction of $- \Delta_{\rm N}$ corresponding to $\mu_2$, without loss of generality real-valued. Then
\begin{align*}
 u := \binom{u_1}{u_2} := \nabla \psi \in \cH_{\rm N}
\end{align*}
is a minimzer of \eqref{eq:minPrincipleNeumann}. The assumption on the direction of the normal vectors can be rewritten as
\begin{align*}
 \nu_1 (x) \nu_2 (x) \leq 0 \quad \text{for almost all}~x \in \partial \Omega,
\end{align*}
and as
\begin{align}\label{eq:normals}
 \nu_1 u_1 + \nu_2 u_2 = 0 \quad \text{almost everywhere on}~\partial \Omega
\end{align}
it follows 
\begin{align}\label{eq:crucial}
 u_1 u_2 \geq 0 \quad \text{almost everywhere on}~\partial \Omega.
\end{align}
In turn, by \eqref{eq:normals}--\eqref{eq:crucial},
\begin{align*}
 v := \binom{v_1}{v_2} := \binom{|u_1|}{|u_2|}
\end{align*}
belongs to $\cH_{\rm N}$. Since, moreover,
\begin{align*}
 \frac{\int_\Omega \big( |\nabla v_1|^2 + |\nabla v_2|^2 \big) - \int_{\partial \Omega} \kappa \big( |v_1|^2 + |v_2|^2 \big)}{\int_\Omega \left( |v_1|^2 + |v_2|^2 \right)}
\end{align*}
is less or equal to the same expression with $v$ replaced by $u$, $v$ is a minimizer of \eqref{eq:minPrincipleNeumann} as well. Hence there exists an eigenfunction $\psi'$ of $- \Delta_{\rm N}$ corresponding to $\mu_2$ such that $\nabla \psi' = v$. But then each component of $\nabla \psi'$ is non-negative, vanishes at each corner of $\partial \Omega$, and 
\begin{align*}
 \Delta \partial_j \psi' = - \mu_2 |u_j| \leq 0, \quad j = 1, 2,
\end{align*}
so that, by the minimum principle for superharmonic functions, $v_j$ equals zero constantly if it vanishes at any point inside $\Omega$, and thus the same is true for $u_j$. However, if one component of $u = \nabla \psi$ is identically zero, say $u_2 = 0$ constantly in $\Omega$, then \eqref{eq:normals} implies 
\begin{align*}
 \partial_1 \psi (x) = u_1 (x) = 0 \quad \text{for almost all}~x \in \partial \Omega~\text{such that}~\nu (x) \neq \e_2.
\end{align*}
If $\Omega$ is not a rectangle, as $\Omega$ does only have convex corners, $\partial \Omega$ must contain a piece $\Sigma$ of positive length such that $\nu (x) \notin \{\e_1, \e_2\}$ for almost all $x \in \Sigma$. Then $\nabla \psi = 0$ identically on $\Sigma$, that is $\psi = \alpha$ identically on $\Sigma$ for some constant $\alpha$. Moreover, as $\Sigma$ is not parallel to the $x_2$-axis and $\partial_2 \psi = u_2 = 0$ constantly in $\Omega$, there exists a nonempty open subset $\cO \subset \Omega$ bordering on $\Sigma$ such that $\psi = \alpha$ constantly on $\cO$. However, this would imply $0 = - \Delta \psi = \mu_2 \psi$ on $\cO$, thus $\psi = 0$ identically on $\cO$ and, by unique continuation, $\psi = 0$ constantly on $\Omega$, a contradiction. Hence if $\Omega$ is not a rectangle then both partial derivatives of $\psi$ are positive everywhere in $\Omega$. This proves assertion (ii).

Let us turn to the proof of assertion (i). If $\Omega$ is a rectangle then an explicit computation of the eigenvalues and eigenfunctions yields that the assertion (i) is true whenever $\Omega$ is not a square. Therefore assume that $\Omega$ is different from a rectangle and that there exist two linearly independent eigenfunctions $\psi, \psi'$ of $- \Delta_{\rm N}$ corresponding to $\mu_2$. Let 
\begin{align*}
 u := \binom{u_1}{u_2} := \nabla \psi \quad \text{and} \quad u' := \binom{u_1'}{u_2'} := \nabla \psi'.
\end{align*}
Then the scalar functions $u_1$ and $u_1'$ are linearly independent. Indeed, if there exist constants $\alpha, \alpha'$ such that $\alpha u_1 + \alpha' u_1' = 0$ constantly then $\partial_1 (\alpha \psi + \alpha' \psi') = 0$ constantly in $\Omega$. Since $\alpha \psi + \alpha' \psi' \in \ker (- \Delta_{\rm N} - \mu_2)$ and $\Omega$ is not a rectangle, it follows that $\alpha \psi + \alpha' \psi' = 0$ constantly in $\Omega$; cf.\ the proof of (ii). However, by linear independence of $\psi$ and $\psi'$, we conclude $\alpha = \alpha' = 0$. Consequently, for any $x_0 \in \Omega$ there exist $\beta, \beta'$ such that $\beta u_1 + \beta' u_1'$ is nontrivial and vanishes at $x_0$. On the other hand, $\beta u + \beta' u'$ is nontrivial and minimizes \eqref{eq:minPrincipleNeumann}. By the same reasoning as in the proof of (ii), this leads to a contradiction. The proof is complete.
\end{proof}

The statement of Theorem \ref{thm:main} (ii) implies explicit assertions on the position of the ``hottest and coldest spots''. For instance the following holds; cf. Figure \ref{fig:cor}.

\begin{corollary}\label{cor:main}
Let $\Omega$ be a lip domain with piecewise $C^\infty$-smooth boundary such that all corners are convex. Let $v_0, v_1$ be the unique leftmost respectively rightmost points of $\partial \Omega$ and assume that $\partial \Omega$ does not contain any piece of positive length parallel to $\mathbf{e}_1 + \mathbf{e}_2$ or $\mathbf{e}_1 - \mathbf{e}_2$. Then the eigenfunction $\psi$ of the Neumann Laplacian corresponding to the first positive eigenvalue $\mu_2$ may be chosen such that $\psi$ takes its minimum only at $v_0$ and its maximum only at $v_1$.
\end{corollary}

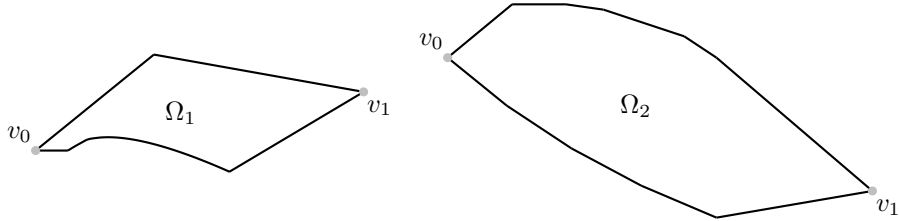
\begin{figure}[h]
\begin{center}
\begin{minipage}{5.3cm}
\begin{tikzpicture}[rotate=-45,transform shape]
\coordinate (a) at (6,-1);
\coordinate (b) at (6.2, 1);
\coordinate (c) at (8,0.6);
\coordinate (d) at (8.5,2.6);
\coordinate (e) at (6.3,-0.7);
\draw[thick] (a) edge (b);
\draw[thick] (b) edge (d);
\draw[thick] (d) edge (c);
\draw[thick] (a) edge (e);
\draw[scale=1,domain=6.3:8,smooth,variable=\x,thick] plot ({\x},{sqrt(\x - 6.3) - 0.7});
\node[rotate=45] at (7,0.7) {$\Omega_1$};
\draw[lightgray, fill] (a) circle (1.5pt);
\node[rotate=45] at (5.7,-1) {$v_0$};
\draw[lightgray, fill] (d) circle (1.5pt);
\node[rotate=45] at (8.8,2.6) {$v_1$};
\end{tikzpicture}
\end{minipage}
\begin{minipage}{6.4cm}
\begin{tikzpicture}[rotate=-45,transform shape]
\coordinate (a) at (6,-1);
\coordinate (a1) at (7,-0.9);
\coordinate (a2) at (8,-0.7);
\coordinate (a3) at (9,-0.4);
\coordinate (b) at (10, 0);
\coordinate (c) at (11,0.5);
\coordinate (d) at (11.2,1.7);
\coordinate (e) at (8.5,1.5);
\coordinate (f) at (8,1.4);
\coordinate (g) at (7,0.9);
\coordinate (h) at (6.6,0.6);
\coordinate (i) at (6.1,0.1);
% \draw[lightgray, fill] (f) circle (1.5pt);
\draw[thick] (a) edge (a1);
\draw[thick] (a1) edge (a2);
\draw[thick] (a2) edge (a3);
\draw[thick] (a3) edge (b);
\draw[thick] (b) edge (d);
% \draw[thick] (c) edge (d);
\draw[thick] (d) edge (e);
\draw[thick] (e) edge (f);
\draw[thick] (f) edge (g);
\draw[thick] (g) edge (h);
\draw[thick] (h) edge (i);
\draw[thick] (i) edge (a);
\node[rotate=45] at (8.2,0.3) {$\Omega_2$};
\draw[lightgray, fill] (a) circle (1.5pt);
\node[rotate=45] at (5.7,-1) {$v_0$};
\draw[lightgray, fill] (d) circle (1.5pt);
\node[rotate=45] at (11.5,1.7) {$v_1$};
\end{tikzpicture}
\end{minipage}
% \begin{minipage}{2.3cm}
% \begin{tikzpicture}[rotate=-45,transform shape]
% \coordinate (a) at (6,-1);
% \coordinate (b) at (7,-0.8);
% \coordinate (c) at (7.1,1);
% % \draw[lightgray, fill] (f) circle (1.5pt);
% \draw[thick] (a) edge (b);
% \draw[thick] (b) edge (c);
% \draw[thick] (c) edge (a);
% \node[rotate=45] at (6.7,-0.4) {$\Omega_3$};
% \draw[lightgray, fill] (a) circle (1.5pt);
% \node[rotate=45] at (5.7,-1) {$v_0$};
% \draw[lightgray, fill] (c) circle (1.5pt);
% \node[rotate=45] at (7.8,1) {$v_1$};
% \end{tikzpicture}
% \end{minipage}
\end{center}
\caption{Examples of domains having a unique ``coldest spot'' $v_0$ and ``hottest spot'' $v_1$ according to Corollary \ref{cor:main}, up to interchanging the two.}
\label{fig:cor}
\end{figure}

\begin{remark}
The approach to the hot spots conjecture via the variational principle proved in this article might work for more general domains. If, for instance, one may guarantee that there is a directional derivative of $\psi$ which does not have any proper sign change on the boundary, say $\partial_1 \psi \geq 0$ on $\partial \Omega$, then it follows from \eqref{eq:minPrincipleNeumann} by the same reasoning as in the above proof that $\partial_1 \psi$ does not change sign in the interior of the domain either and then even $\partial_1 \psi > 0$ in $\Omega$; in particular, $\psi$ has no critical point in $\Omega$. 
\end{remark}

\begin{remark}
Above we have used, for convenience, the well-known fact that $\mu_2 < \lambda_1$ is true, so that the first eigenvalue $\eta_1$ of the operator $A$ equals $\mu_2$. However, for lip domains this can be concluded with the help of the variational principle \eqref{eq:minMax}. In fact, assuming $\eta_1 = \lambda_1$ and reasoning as in the proof of Theorem \ref{thm:main} would yield that the components of $\nabla^\perp \phi$ for the corresponding eigenfunction $\phi$ of $- \Delta_{\rm D}$ are non-negative in $\Omega$. Thus the function $\phi$ is monotonous in both coordinate directions. As $\phi$ equals zero constantly on $\partial \Omega$, this implies $\phi = 0$ constantly, a contradiction.
\end{remark}

%\appendix
%
%
%\section{Technicalities in Sobolev spaces}
%
%
%\begin{lemma}\label{lem:langweiligerRamsch2}
%On any Lipschitz domain, for any $\phi \in \dom (- \Delta_{\rm D})$, $\langle \nabla^\perp \phi, \nu \rangle = 0$ on $\partial \Omega$.
%\end{lemma}
%
%\begin{proof}
%If $\phi \in \dom (- \Delta_{\rm D})$ then for each $\psi \in H^1 (\Omega)$,
%\begin{align*}
 %\int_{\partial \Omega} \left\langle \nabla^\perp \phi, \nu \right\rangle \overline{\psi} = \int_\Omega \left\langle \nabla^\perp \phi, \nabla \psi \right\rangle + \int_\Omega \diver \nabla^\perp \phi \overline{\psi},
%\end{align*}
%and the right-hand side vanishes by orthogonality of $\nabla^\perp H_0^1 (\Omega)$ and $\nabla H^1 (\Omega)$ together with $\diver \nabla^\perp \phi = 0$.
%\end{proof}

\section*{Acknowledgements}
The author gratefully acknowledges financial support by the grant no.\ 2018-04560 of the Swedish Research Council (VR). Furthermore, he wishes to thank all people with whom he had conversations about the hot spots conjecture in the past. These include James B.\ Kennedy, David Krej\v{c}i\v{r}\'ik, Vladimir Lotoreichik and Sugata Mondal.

\end{document}